\documentclass[preprint,12pt,nopreprintline]{elsarticle}

\usepackage[T1]{fontenc}
\usepackage{lmodern}
\usepackage{amsmath,amssymb,amsthm,mathtools}
\usepackage{mathrsfs}
\usepackage{enumitem}
\usepackage{microtype}
\usepackage[colorlinks=true,
            linkcolor=blue,
            citecolor=blue,
            urlcolor=blue]{hyperref}
\biboptions{numbers,sort&compress}
\journal{Journal of Functional Analysis}
\makeatletter
\let\ps@pprintTitle\ps@plain
\makeatother

\newtheorem{theorem}{Theorem}[section]
\newtheorem{proposition}[theorem]{Proposition}
\newtheorem{lemma}[theorem]{Lemma}
\newtheorem{corollary}[theorem]{Corollary}
\theoremstyle{definition}
\newtheorem{definition}[theorem]{Definition}
\theoremstyle{remark}
\newtheorem{remark}[theorem]{Remark}
\numberwithin{equation}{section}

\begin{document}
\begin{frontmatter}

\title{Transfinite contractive propagation and the ball fixed point property in \texorpdfstring{$C(K)$}{C(K)}}
\author{Cleon S. Barroso}
 \address{Department of Mathematics\\ Federal University of Ceará, Fortaleza, 60455-360, CE, Brazil}
\ead{cleonbar@mat.ufc.br}

\begin{abstract}
We prove that, for a compact Hausdorff space $K$, the real Banach space $C(K)$ has the ball fixed point property if and only if $K$ is extremally disconnected. The new implication is obtained by constructing a fixed-point-free nonexpansive self-map of the whole closed unit ball whenever $K$ is a compact $F$-space which is not extremally disconnected. The construction uses a transfinite contractive propagation system with two increasing profiles. A successor--limit delay preserves their tail constraint and has no subsolution in the profile domain. Analysis records boundary deficits; positive operators on ordinal $C_0$-spaces and a common center determined by the two tails yield a nonexpansive synthesis satisfying an order domination inequality. The required topological families are obtained from a gap in a maximal Boolean chain in the zero-dimensional case, and from a transfinite extension of signs otherwise. The argument works in ZFC and resolves the question posed by Avil\'es, Jap\'on, Lennard, Mart\'inez-Cervantes, and Stawski.
\end{abstract}


\end{frontmatter}

\section{Introduction}
\label{sec:int}

The fixed point theory of nonexpansive mappings on \(C(K)\)-spaces exhibits a particularly subtle interaction between the topology of the compact space \(K\) and the metric geometry of \(C(K)\). Recall that a mapping $T:C\to C$ on a subset \(C\) of a Banach space is \emph{nonexpansive} if
\[
\|Tx-Ty\|\leq \|x-y\|,\qquad x,y\in C.
\]

The ball fixed point problem for spaces of continuous functions asks which compact spaces $K$ have the property that every nonexpansive self-map of the closed unit ball of $C(K)$ has a fixed point. We prove that these are exactly the extremally disconnected compact spaces. The proof constructs a nonexpansive map on the whole ball from transfinite profiles that record incompatible boundary values. Throughout the paper, $K$ is compact Hausdorff and $C(K)$ is the real Banach space of continuous functions on $K$, with $\|f\|_\infty=\sup_{t\in K}|f(t)|$. We write $B_X$ for the closed unit ball of a Banach space $X$. A map between metric spaces is \emph{nonexpansive} if its Lipschitz constant is at most $1$. We work in ZFC and identify cardinals with their initial ordinals.

\subsection{The ball fixed point problem and the main result}

\begin{definition}
A Banach space $X$ has the \emph{ball fixed point property} (BFPP) if every nonexpansive map $T:B_X\to B_X$ has a fixed point.
Equivalently, the same assertion holds on every closed ball of $X$, by translation and dilation; a ball of radius zero is immediate.
\end{definition}

Recall that $K$ is \emph{extremally disconnected} if the closure of every open subset is open. A set of the form
$\operatorname{coz}(f)=\{t:f(t)\ne0\}$, for $f\in C(K)$, is a \emph{cozero set}. We use the characterization of a compact
$F$-space by the property that disjoint cozero sets have disjoint closures; see \cite[Lemma~3.1]{AJLCS}. Avil\'es, Jap\'on, Lennard, Mart\'inez-Cervantes, and Stawski proved
\begin{equation}
 K\text{ extremally disconnected}
 \ \Longrightarrow\ C(K)\text{ has BFPP}
 \ \Longrightarrow\ K\text{ is an }F\text{-space}.
 \label{eq:known-implications}
\end{equation}
See \cite[Corollary~2.4 and Theorem~3.2]{AJLCS}. Their work provides examples of compact $F$-spaces which fail the
BFPP and asks whether the first implication is an equivalence. The unresolved class in that work consists of compact $F$-spaces
which are not extremally disconnected. We settle this case.

\begin{theorem}[Main result]
\label{thm:intro-main}
For every compact Hausdorff space $K$, the real space $C(K)$ has the BFPP if and only if $K$ is extremally disconnected.
\end{theorem}

The constructive part is Theorem~\ref{tcp:thm:main}. For each compact $F$-space $K$ which is not extremally disconnected, it gives an uncountable regular cardinal $\kappa$, a nonempty closed bounded convex set $\mathcal D\subset\ell_\infty(\kappa)$, and nonexpansive maps
\[
 \mathcal A:B_{C(K)}\to\mathcal D,\qquad
 \mathscr P:\mathcal D\to\mathcal D,\qquad
 \mathcal J:\mathcal D\to B_{C(K)}
\]
such that $\mathscr P$ is a transfinite propagator and $\operatorname{Fix}(\mathcal A\mathcal J\mathscr P)=\varnothing$.
Consequently $T=\mathcal J\mathscr P\mathcal A$ is a fixed-point-free nonexpansive self-map of $B_{C(K)}$.
Together with \eqref{eq:known-implications}, this proves Theorem~\ref{thm:intro-main}; the final deduction is recorded in
Corollary~\ref{tcp:cor:characterization}.

\subsection{From discrete propagation to transfinite realization}

In \cite{Barroso2026}, contractive propagation systems were introduced as an abstract framework for constructing fixed-point-free nonexpansive maps on bounded, closed, and convex subsets of Banach spaces. Crucially, these discrete systems decouple the coefficient dynamics, the propagation mechanism, and the comparison argument that rules out fixed points. A typical recurrence takes the form
\begin{equation}
 \mathcal S_{n+1}(x)
 =a_n\mathcal S_n(x)+b_ng_n(x)+c_n.
 \label{tcp:eq:discrete}
\end{equation}
If the coefficients and the translation term are independent of $x$, then
\[
 \operatorname{Lip}(\mathcal S_{n+1})
 \le |a_n|\operatorname{Lip}(\mathcal S_n)
      +|b_n|\operatorname{Lip}(g_n).
\]
Thus nonexpansive constituent maps and $|a_n|+|b_n|\le1$ give a nonexpansive successor update. Here and below, \emph{contractive}
means Lipschitz constant at most $1$, not strictly less than $1$.

\smallskip 

Successor--limit shifts with upper and lower envelope rules already occur in \cite[proof of Theorem~4.1]{AJLCS}. The contribution here is a realization of transfinite propagation on the unit ball of $C(K)$ for every compact $F$-space which is not extremally disconnected. Two monotone profiles carry the boundary information. Their delay $P_0$ preserves an invariant domain $\mathcal D_0$ and has no subsolution $z\le P_0z$ in that domain. The main analytic step constructs nonexpansive maps $\mathcal A_0,\mathcal J_0$ with
\[
 \mathcal A_0\mathcal J_0z\le z\qquad(z\in\mathcal D_0).
\]
A fixed point of $\mathcal A_0\mathcal J_0P_0$ would therefore be such a subsolution. Continuity of $\mathcal J_0$ is obtained from positive operators on ordinal $C_0$-spaces; a convex integral representation gives the Lipschitz constant $1$.

\smallskip 

The paper is organized as follows. Section~\ref{sec:tcp} defines the propagation framework and proves
the transfer and comparison principles. Section~\ref{sec:realization} constructs the two-profile realization under explicit topological hypotheses. Section~\ref{sec:families} verifies these hypotheses in the zero-dimensional and non-zero-dimensional cases. Section~\ref{sec:conclusion} encodes the two profiles in one ordinal and completes the BFPP characterization.

\section{Transfinite contractive propagation systems}
\label{sec:tcp}

The metric and order arguments in this section are independent of the later topological constructions. Their role is to transfer the absence of profile subsolutions to the absence of fixed points in the ball.

\subsection{Ordinal envelopes and transfinite propagators}
\label{tcp:subsec:envelopes}

For an ordinal~$\kappa$, let
\[
 \ell_\infty(\kappa)=\left\{x=(x_\alpha)_{\alpha<\kappa}:\sup_{\alpha<\kappa}|x_\alpha|<\infty\right\},\qquad \|x\|_\infty=\sup_{\alpha<\kappa}|x_\alpha|.
\]
All inequalities between profiles are understood coordinatewise. We begin with the two basic envelope operators.

\begin{definition}[Ordinal envelopes]
\label{tcp:def:envelopes}
For a nonzero limit ordinal~$\lambda$ and $a\in\ell_\infty(\lambda)$, define
\[
 \overline{\mathcal E}_\lambda(a)=\limsup_{\alpha\uparrow\lambda}a_\alpha:=\inf_{\beta<\lambda}\sup_{\beta\leq\alpha<\lambda}a_\alpha
\]
and
\[
 \underline{\mathcal E}_\lambda(a)=\liminf_{\alpha\uparrow\lambda}a_\alpha:=\sup_{\beta<\lambda}\inf_{\beta\leq\alpha<\lambda}a_\alpha.
\]
\end{definition}

These operators have exactly the metric property required by a contractive propagation theory.

\begin{lemma}[Envelope contraction]
\label{tcp:lem:envelope}
For every nonzero limit ordinal~$\lambda$, both maps 
\[
\overline{\mathcal E}_\lambda,\,\underline{\mathcal E}_\lambda:\ell_\infty(\lambda)\longrightarrow\mathbb R 
\]
are nonexpansive.
\end{lemma}

\begin{proof}
Let $d=\|x-y\|_\infty$. For every $\beta<\lambda$,
\[
 \sup_{\beta\le\alpha<\lambda}x_\alpha\le \sup_{\beta\le\alpha<\lambda}y_\alpha+d.
\]
Taking the infimum over~$\beta$ and then interchanging $x,y$ proves
\[
|\overline{\mathcal E}_\lambda(x)-\overline{\mathcal E}_\lambda(y)|\le d.
\]
The identity $\underline{\mathcal E}_\lambda(x) =-\overline{\mathcal E}_\lambda(-x)$ gives the other estimate.
\end{proof}

\begin{definition}[Transfinite contractive propagator]
\label{tcp:def:propagator}
Let $\kappa>0$ be an ordinal and let $\mathcal D\subset\ell_\infty(\kappa)$ be nonempty. A \emph{transfinite contractive propagator}, abbreviated TCP, is a self-mapping
\[
\mathscr P:\mathcal D\longrightarrow\mathcal D 
\]
whose zeroth coordinate is a fixed constant, whose successor coordinates are given by nonexpansive rules
\[
 (\mathscr Px)_{\alpha+1}=\Phi_\alpha\bigl(x|_{[0,\alpha]}\bigr),\qquad|\Phi_\alpha(u)-\Phi_\alpha(v)|\le\|u-v\|_\infty,
\]
and whose nonzero limit coordinates satisfy
\[
 (\mathscr Px)_\lambda
 =\mathcal E^{\varepsilon_\lambda}_\lambda
      (x|_{[0,\lambda)}),
 \qquad \varepsilon_\lambda\in\{-,+\},
\]
where $\mathcal E^-_\lambda=\underline{\mathcal E}_\lambda$ and $\mathcal E^+_\lambda=\overline{\mathcal E}_\lambda$.
The successor rules and the choice of $\varepsilon_\lambda$ are fixed independently of the input. At a successor of a limit, a
successor rule may store the other envelope of the preceding prefix.
\end{definition}

\begin{remark}
\label{tcp:rem:propagator-metric}
The successor rules need only be defined on the prefix sets arising from~$\mathcal D$. Each output coordinate is a nonexpansive function of the input profile, so Lemma~\ref{tcp:lem:envelope} implies
\[
 \|\mathscr Px-\mathscr Py\|_\infty\le\|x-y\|_\infty.
\]
The requirement $\mathscr P(\mathcal D)\subseteq\mathcal D$ is a separate invariance condition. The recursion prescribes the coordinates of one image profile from prefixes of its input; it does not assert convergence of the iterates of~$\mathscr P$.
\end{remark}

\subsection{Nonexpansive realization and fixed point transfer}
\label{tcp:subsec:transfer}

\begin{definition}[Contractive profile realization]
\label{tcp:def:realization}
Let $X$ be a Banach space and let $\mathscr P:\mathcal D\to\mathcal D$ be a TCP. A \emph{contractive profile realization} in~$B_X$ consists of nonexpansive maps
\[
 \mathcal A:B_X\longrightarrow\mathcal D,\qquad \mathcal J:\mathcal D\longrightarrow B_X.
\]
They will be called the analysis and synthesis maps, respectively. The data $(\mathcal D,\mathscr P,\mathcal A,\mathcal J)$ form a \emph{transfinite contractive propagation system} on $B_X$. The realization is \emph{fixed-point-free} if $\operatorname{Fix}(\mathcal A\mathcal J\mathscr P)=\varnothing$.
\end{definition}

\begin{proposition}[Transfer principle]
\label{tcp:prop:transfer}
For any contractive profile realization, the composition map $T=\mathcal J\mathscr P\mathcal A:B_X\to B_X$ is nonexpansive.
Moreover, $\mathcal A$ and $\mathcal J\mathscr P$ restrict to mutually inverse bijections between
\[
 \operatorname{Fix}(T) \quad\text{and}\quad \operatorname{Fix}(\mathcal A\mathcal J\mathscr P).
\]
In particular, a fixed-point-free realization yields a fixed-point-free nonexpansive self-map of~$B_X$.
\end{proposition}

\begin{proof}
The metric assertion follows from
\[
 \|Tx-Ty\|
 \le\|\mathscr P\mathcal Ax-\mathscr P\mathcal Ay\|_\infty
 \le\|\mathcal Ax-\mathcal Ay\|_\infty
 \le\|x-y\|.
\]
If $Tx=x$, then $\mathcal Ax= \mathcal A\mathcal J\mathscr P\mathcal Ax$. Conversely, if $z=\mathcal A\mathcal J\mathscr Pz$, then
$x=\mathcal J\mathscr Pz$ satisfies $Tx=x$ and $\mathcal Ax=z$. On $\operatorname{Fix}(T)$ one also has $\mathcal J\mathscr P\mathcal Ax=x$, proving the bijection.
\end{proof}

\begin{lemma}[Order domination criterion]
\label{tcp:lem:domination}
In the setting of Definition~\ref{tcp:def:realization}, suppose that
\[
\mathcal A\mathcal Jz\le z
\] 
for every $z\in\mathcal D$, and
\[
 \{z\in\mathcal D:z\le\mathscr Pz\}=\varnothing.
\]
Then $\operatorname{Fix}(\mathcal A\mathcal J\mathscr P) =\varnothing$.
\end{lemma}

\begin{proof}
A fixed point $z=\mathcal A\mathcal J\mathscr Pz$ would satisfy $z\le\mathscr Pz$.
\end{proof}

\begin{remark}
The extra hypothesis $\operatorname{Fix}(\mathcal A\mathcal J\mathscr P)=\varnothing$ in the last assertion of Proposition~\ref{tcp:prop:transfer} is essential. The transfer implication itself holds for arbitrary maps with the indicated domains; it cannot replace a proof that $\mathcal A\mathcal J\mathscr P$ has no fixed point. Likewise, the absence of fixed points for~$\mathscr P$ alone does not provide that proof. Our construction will establish the stronger order estimate required by Lemma~\ref{tcp:lem:domination}.
\end{remark}

\section{Two-profile realization}
\label{sec:realization}

We first construct the ordinal operators needed for synthesis. The realization theorem will then apply to any compact space
carrying the two families specified below.

\subsection{Monotone profiles, delay, and regularization}
\label{tcp:subsec:profiles}

For an infinite limit ordinal~$\tau$, put
\[
 \mathcal M_\tau=\big\{a\in[0,2]^\tau:a_0=0,\ a_\alpha\le a_\beta \text{ whenever }\alpha\le\beta\big\},
 \qquad \ell_\tau(a)=\sup_{\alpha<\tau}a_\alpha.
\]
The set $\mathcal M_\tau$ is closed and convex in $\ell_\infty(\tau)$. Its tail functional $\ell_\tau$ is nonexpansive and affine on
$\mathcal M_\tau$. Indeed, for $a,b\in\mathcal M_\tau$ and $0\le\theta\le1$,
\[
 \ell_\tau(\theta a+(1-\theta)b)
 =\theta\ell_\tau(a)+(1-\theta)\ell_\tau(b).
\]
The inequality $\le$ holds for every coordinate. Given $\epsilon>0$, choose one index beyond indices where $a$ and $b$ are within
$\epsilon$ of their respective suprema. Evaluation there gives the reverse inequality after letting $\epsilon\downarrow0$. Affinity is asserted only on these increasing profiles.

\begin{lemma}[Delay and regularization]
\label{tcp:lem:profile-operators}
For $a\in\mathcal M_\tau$, define
\begin{equation}
 (Q_\tau a)_0=0,\qquad
 (Q_\tau a)_{\alpha+1}=a_\alpha,\qquad
 (Q_\tau a)_\lambda=\sup_{\alpha<\lambda}a_\alpha
 \quad(0<\lambda<\tau\text{ a limit}),
 \label{tcp:eq:delay}
\end{equation}
and
\begin{equation}
 (R_\tau a)_\alpha=
 \begin{cases}
  \sup_{\beta<\alpha}a_\beta,
       &\alpha\text{ a nonzero limit},\\
  a_\alpha,&\alpha=0\text{ or a successor}.
 \end{cases}
 \label{tcp:eq:regularization}
\end{equation}
Then $Q_\tau,R_\tau$ are nonexpansive self-maps of $\mathcal M_\tau$, both preserve~$\ell_\tau$, and both lie
coordinatewise below the identity. Furthermore:
\begin{enumerate}
 \item $a\le Q_\tau a$ implies $a=0$;
 \item $R_\tau a$, regarded as a real-valued function on the ordinal space $[0,\tau)$ with its order topology, is continuous.
\end{enumerate}
\end{lemma}

\begin{proof}
Each coordinate in the two formulas is a constant, a coordinate evaluation, or a supremum, and is therefore nonexpansive. The
monotonicity of~$a$ gives $Q_\tau a\le a$ and $R_\tau a\le a$. For $\alpha<\beta$, one has $(Q_\tau a)_\alpha\le a_\alpha\le(Q_\tau a)_\beta$; hence $Q_\tau a$ is increasing. The same conclusion for $R_\tau a$ follows directly from its definition. Every coordinate~$a_\alpha$ occurs at the successor~$\alpha+1<\tau$ in $Q_\tau a$, proving preservation of the supremum. Successor ordinals are cofinal in~$\tau$, and $R_\tau$ leaves their coordinates unchanged, which proves the corresponding assertion for~$R_\tau$.

If $a\le Q_\tau a$, transfinite induction gives $a_\alpha=0$: the assertion starts at~$0$, is inherited from the predecessor at a
successor, and follows from the supremum of the previous zero coordinates at a limit. For continuity, the only nonisolated points
of $[0,\tau)$ are its nonzero limits. At such a point~$\lambda$, cofinality of the successors gives
\[
 \sup_{\beta<\lambda}(R_\tau a)_\beta
 =\sup_{\beta<\lambda}a_\beta=(R_\tau a)_\lambda.
\]
Since $R_\tau a$ is increasing, its values on $(\beta,\lambda]$ approach its value at $\lambda$. Such intervals are neighborhoods of $\lambda$ in $[0,\tau)$, which proves continuity.
\end{proof}

\subsection{Positive operators on ordinal spaces}
\label{subsec:positive-operators}

We write $C_0([0,\tau))$ for the continuous functions on this locally compact ordinal space which vanish at infinity. Every interval $[0,\alpha]$, $\alpha<\tau$, is compact and clopen. These intervals form an open cover of $[0,\tau)$, so every compact subset is contained in one of them, by a finite subcover.

\begin{lemma}[Positive synthesis from an increasing family]
\label{tcp:lem:positive-synthesis}
Let $(w_\alpha)_{\alpha<\tau}\subset C(K,[0,1])$ be increasing.
There is a unique positive linear contraction
\[
 S_w:C_0([0,\tau))\longrightarrow C(K),\qquad
 S_w\mathbf1_{[0,\alpha]}=w_\alpha.
\]
For each $t\in K$, its evaluation at~$t$ is represented by a finite
positive Radon measure~$\mu_t$ such that
\begin{equation}
 \mu_t([0,\alpha])=w_\alpha(t),\qquad
 \mu_t([0,\tau))=\sup_{\alpha<\tau}w_\alpha(t)\le1.
 \label{tcp:eq:representing-mass}
\end{equation}
\end{lemma}

\begin{proof}
Put $e_\alpha=\mathbf1_{[0,\alpha]}$. These functions are linearly independent: in a finite relation, evaluation on successive intervals determines all tail sums of the coefficients and hence all coefficients. Their span $V$ is a subalgebra and a sublattice of $C_0([0,\tau))$: its elements are step functions with finitely many breakpoints and bounded support, and products, maxima, and minima retain that form. If $\alpha<\beta$, then $e_\alpha$ separates them; also $e_\alpha(\alpha)=1$. Thus $V$ separates points and vanishes at no point. The locally compact Stone--Weierstrass theorem makes~$V$ uniformly dense. Prescribe $S_we_\alpha=w_\alpha$ on~$V$. For $\alpha_1<\cdots<\alpha_n$, set
\[
 z=\sum_{i=1}^n c_i e_{\alpha_i},\qquad
 q_i=\sum_{j=i}^n c_j.
\]
The values of $z$ on its successive intervals are $q_1,\ldots,q_n$, and $z$ is zero afterwards. Thus $\|z\|_\infty=\max_i|q_i|$.
Summation by parts gives
\[
 S_wz=q_1w_{\alpha_1}
       +\sum_{i=2}^n q_i(w_{\alpha_i}-w_{\alpha_{i-1}}).
\]
For each $t\in K$, the multipliers of the~$q_i$ are nonnegative and sum to $w_{\alpha_n}(t)\le1$. Hence
$\|S_wz\|_\infty\le\|z\|_\infty$; if $z\ge0$, all~$q_i$ are nonnegative and $S_wz\ge0$. Density gives the unique contractive
extension. Positivity persists because~$V$ is a sublattice: the positive parts of step approximants approximate any nonnegative
member of~$C_0([0,\tau))$.

The Riesz representation theorem supplies~$\mu_t$. Evaluation at $e_\alpha$ yields the first equality in
\eqref{tcp:eq:representing-mass}. Inner regularity, together with the fact that every compact subset is contained in some $[0,\alpha]$, gives the second equality. In particular, the mass identity uses inner regularity, not continuity from below for an uncountable union.
\end{proof}

\subsection{A realization theorem with two monotone profiles}
\label{tcp:subsec:two-profile-realization}

\begin{definition}[An inseparable pair of increasing families]
\label{tcp:def:inseparable-pair}
Let $\rho,\sigma$ be infinite cardinals. An inseparable pair on~$K$ consists of increasing families
\[
 (u_\alpha)_{\alpha<\rho},\quad
 (v_\beta)_{\beta<\sigma}\quad\text{in }C(K,[0,1])
\]
satisfying
\begin{equation}
 u_0=v_0=0,\qquad u_\alpha v_\beta=0
 \quad(\alpha<\rho,\ \beta<\sigma),
 \label{tcp:eq:orthogonal-families}
\end{equation}
and, with $F_\alpha=\{u_\alpha=1\}$ and
$G_\beta=\{v_\beta=1\}$,
\begin{equation}
 \overline{\bigcup_{\alpha<\rho}F_\alpha}
 \cap
 \overline{\bigcup_{\beta<\sigma}G_\beta}\ne\varnothing.
 \label{tcp:eq:terminal-intersection}
\end{equation}
No continuity of either family with respect to its ordinal index is required.
\end{definition}

\begin{theorem}[Two-profile realization]
\label{tcp:thm:two-profile-realization}
Suppose that a compact Hausdorff space~$K$ admits an inseparable pair as in Definition~\ref{tcp:def:inseparable-pair}. There exist a nonempty closed bounded convex set
\[
 \mathcal D_0\subset \ell_\infty(\rho)\times\ell_\infty(\sigma)
\]
with the maximum norm, and nonexpansive maps
\[
 \mathcal A_0:B_{C(K)}\to\mathcal D_0,
 \qquad
 \mathcal J_0:\mathcal D_0\to B_{C(K)},
 \qquad
 P_0:\mathcal D_0\to\mathcal D_0,
\]
such that
\[
 \mathcal A_0\mathcal J_0z\le z\quad(z\in\mathcal D_0),
 \qquad
 \operatorname{Fix}(\mathcal A_0\mathcal J_0P_0)=\varnothing.
\]
Consequently $\mathcal J_0P_0\mathcal A_0$ is a fixed-point-free nonexpansive self-map of~$B_{C(K)}$.
\end{theorem}

\begin{proof}
Define
\begin{equation}
 \mathcal D_0=
 \{(a,b)\in\mathcal M_\rho\times\mathcal M_\sigma:
       \ell_\rho(a)+\ell_\sigma(b)\ge2\}.
 \label{tcp:eq:profile-domain}
\end{equation}
It is nonempty: take both profiles equal to~$1$ outside their zero coordinate. Closedness follows from the continuity of the tail
functionals, convexity from their affinity on increasing profiles, and boundedness from the coordinate bounds. Set
\begin{equation}
 P_0(a,b)=(Q_\rho a,Q_\sigma b).
 \label{tcp:eq:product-propagator}
\end{equation}
Lemma~\ref{tcp:lem:profile-operators} proves that this map is nonexpansive and preserves~$\mathcal D_0$.

\smallskip\noindent
\emph{The analysis map.} For $f\in B_{C(K)}$, define
\begin{equation}
 a_\alpha(f)=\sup_{t\in F_\alpha}(1-f(t)),\qquad
 b_\beta(f)=\sup_{t\in G_\beta}(1+f(t)),\qquad
 \mathcal A_0f=(a(f),b(f)),
 \label{tcp:eq:analysis}
\end{equation}
where a supremum over an empty set is assigned the value~$0$ in these two formulas. Both deficits are nonnegative and at most~$2$.
Nestedness of the level sets gives increasing profiles, and their initial coordinates are zero. Each coordinate functional is
nonexpansive, hence so is~$\mathcal A_0$. Choose $p$ in the intersection in \eqref{tcp:eq:terminal-intersection}. Continuity of~$f$ gives
\[
 \ell_\rho(a(f))\ge1-f(p),\qquad
 \ell_\sigma(b(f))\ge1+f(p).
\]
Their sum is at least~$2$, so $\mathcal A_0(B_{C(K)})\subseteq\mathcal D_0$.

\smallskip\noindent
\emph{The synthesis map.} Apply Lemma~\ref{tcp:lem:positive-synthesis} to obtain positive linear contractions
\[
 S_+:C_0([0,\rho))\to C(K),\qquad
 S_-:C_0([0,\sigma))\to C(K)
\]
with $S_+\mathbf1_{[0,\alpha]}=u_\alpha$ and $S_-\mathbf1_{[0,\beta]}=v_\beta$. Let their representing measures
at~$t$ be $\mu_t,\nu_t$, with total masses $m_t,n_t$. By \eqref{tcp:eq:representing-mass} and
\eqref{tcp:eq:orthogonal-families},
\begin{equation}
 m_t=\sup_\alpha u_\alpha(t),\qquad
 n_t=\sup_\beta v_\beta(t),\qquad m_t+n_t\le1.
 \label{tcp:eq:mass-bound}
\end{equation}
Indeed, if $m_t>0$, some $u_\alpha(t)>0$, and orthogonality forces every $v_\beta(t)$ to vanish, so $n_t=0$. Interchanging the families gives the other case, and each mass is at most~$1$.

For $z=(a,b)\in\mathcal D_0$, write
\begin{equation}
 A=\ell_\rho(a),\qquad B=\ell_\sigma(b),\qquad
 c(z)=\frac{B-A}{2}.
 \label{tcp:eq:center}
\end{equation}
Since $0\le A,B\le2$ and $A+B\ge2$, we have $c(z)\in[-1,1]$ and
\begin{equation}
 1-A\le c(z)\le B-1.
 \label{tcp:eq:center-bounds}
\end{equation}
Furthermore, if $\|z-z'\|\le d$, then $|A-A'|,|B-B'|\le d$, and consequently $|c(z)-c(z')|\le d$.

With $r^+=\max\{r,0\}$, define the continuous ordinal functions
\begin{equation}
 x_+(z)=(1-R_\rho a-c(z))^+, \qquad x_-(z)=(1-R_\sigma b+c(z))^+.
 \label{tcp:eq:synthesis-profiles}
\end{equation}
They are nonincreasing, and their tail limits are respectively $(1-A-c(z))^+$ and $(1-B+c(z))^+$, both equal to
\[
 \left(1-\frac{A+B}{2}\right)^+=0.
\]
Thus $x_+(z)\in C_0([0,\rho))$ and $x_-(z)\in C_0([0,\sigma))$. Here vanishing on a tail up to an arbitrary positive tolerance implies vanishing at infinity, since closed initial ordinal intervals are compact. Set
\begin{equation}
\mathcal J_0z=c(z)\mathbf1+S_+x_+(z)-S_-x_-(z).
 \label{tcp:eq:synthesis}
\end{equation}
This formula defines a continuous function on $K$, since both operator arguments belong to their stated $C_0$-domains. No
continuity of the total-mass functions $t\mapsto m_t,n_t$ is needed.

\smallskip\noindent
\emph{The exact nonexpansive estimate.} Using the representing measures, expand \eqref{tcp:eq:synthesis} at a fixed $t\in K$ as
\begin{align}
 (\mathcal J_0z)(t)
 ={}&(1-m_t-n_t)c(z)\notag\\
 &+\int_{[0,\rho)}
       \max\{c(z),1-(R_\rho a)_\alpha\}\,d\mu_t(\alpha)\notag\\
 &+\int_{[0,\sigma)}
       \min\{c(z),(R_\sigma b)_\beta-1\}\,d\nu_t(\beta).
 \label{tcp:eq:convex-integral}
\end{align}
The integrands are bounded continuous functions, so these integrals are well defined for the finite measures. Each displayed scalar value lies in $[-1,1]$. The coefficients and measures are nonnegative and have total weight $(1-m_t-n_t)+m_t+n_t=1$.
It follows that $|(\mathcal J_0z)(t)|\le1$.

If $\|z-z'\|\le d$, the center and every coordinate of the two regularized profiles change by at most~$d$. For real numbers,
\[
 |\max\{r,s\}-\max\{r',s'\}|
 \le\max\{|r-r'|,|s-s'|\},
\]
and the same estimate holds for~$\min$. Therefore every integrand in \eqref{tcp:eq:convex-integral} changes by at most~$d$.
The weights are independent of~$z$, so
\[
 |(\mathcal J_0z)(t)-(\mathcal J_0z')(t)|\le d.
\]
Taking the supremum over~$t$ proves that $\mathcal J_0:\mathcal D_0\to B_{C(K)}$ is nonexpansive.

\smallskip\noindent
\emph{The order estimate.} Fix $\alpha<\rho$ and $t\in F_\alpha$. Then $\mu_t([0,\alpha])=u_\alpha(t)=1$. In view of \eqref{tcp:eq:mass-bound}, $m_t=1$, $n_t=0$, and $\mu_t$ is a probability measure concentrated on $[0,\alpha]$. For $\gamma\le\alpha$,
\[
(R_\rho a)_\gamma\le a_\gamma\le a_\alpha.
\]
Equation~\eqref{tcp:eq:convex-integral} now yields $(\mathcal J_0z)(t)\ge1-a_\alpha$. Similarly, for $t\in G_\beta$, $\nu_t$ is a probability measure concentrated on $[0,\beta]$ and $(\mathcal J_0z)(t)\le b_\beta-1$. Taking the suprema in \eqref{tcp:eq:analysis}, including the prescribed empty-set convention, gives
\begin{equation}
\mathcal A_0\mathcal J_0(a,b)\le(a,b).
 \label{tcp:eq:order-realization}
\end{equation}

\smallskip\noindent
\emph{Exclusion of fixed points.} If $z=\mathcal A_0\mathcal J_0P_0z$, then \eqref{tcp:eq:order-realization} implies $z\le P_0z$. Thus $a\le Q_\rho a$ and $b\le Q_\sigma b$. Lemma~\ref{tcp:lem:profile-operators} gives $a=b=0$, contradicting
$\ell_\rho(a)+\ell_\sigma(b)\ge2$. Finally, the proof of Proposition~\ref{tcp:prop:transfer}, which applies equally to the product profile space, shows that $\mathcal J_0P_0\mathcal A_0$ is fixed-point-free.
\end{proof}

\begin{remark}
The use of two tails in \eqref{tcp:eq:center} allows $\rho$ and $\sigma$ to be different. The terminal obstruction supplies the
inequality $A+B\ge2$, without requiring either tail separately to equal~$1$. The resulting common center ensures continuity of the
synthesis, and its convex integral representation preserves the Lipschitz constant~$1$.
\end{remark}

\section{Inseparable families on compact \texorpdfstring{$F$}{F}-spaces}
\label{sec:families}
\label{tcp:subsec:families}

We now verify the hypotheses of Theorem~\ref{tcp:thm:two-profile-realization}. A compact space is
\emph{zero-dimensional} if it has a base of clopen sets. The two topological cases will provide families of possibly different lengths.

\subsection{Countable unions of cozero sets}
\label{tcp:subsec:topological-preliminaries}

\begin{lemma}
\label{tcp:lem:countable-cozero}
A countable union of cozero sets in a compact Hausdorff space is a cozero set. Consequently, in an $F$-space, two disjoint countable unions of cozero sets have disjoint closures.
\end{lemma}

\begin{proof}
If $U_n=\operatorname{coz}(f_n)$, the uniformly convergent series
\[
 f=\sum_{n=1}^\infty 2^{-n}\frac{|f_n|}{1+\|f_n\|_\infty}
\]
defines a continuous nonnegative function with $\operatorname{coz}(f)=\bigcup_n U_n$. The last assertion follows from the $F$-space property.
\end{proof}

\subsection{Boolean gaps and the zero-dimensional case}

We recall the Boolean-algebra terminology used below; see
\cite{Koppelberg} for background. A \emph{Boolean algebra} is a bounded
complemented distributive lattice, with meet $a\wedge b$, join
$a\vee b$, complement $\neg a$, and least and greatest elements
$0,1$. Its order is given by $a\le b$ if $a\wedge b=a$.
It is \emph{complete} if every subset $S$ has a least upper bound,
denoted by $\bigvee S$; the empty join is~$0$.
For a compact space~$K$, the clopen subsets form the Boolean algebra
$\operatorname{Clop}(K)$ under intersection, finite union, and
complementation in~$K$, with order given by inclusion.
For an arbitrary family of clopens, its Boolean supremum, when it
exists, means its least clopen upper bound and need not equal its
set-theoretic union.

A \emph{chain} is a subset linearly ordered by $\le$, and a
\emph{maximal chain} is one maximal under inclusion. A subset $C$
of a linearly ordered set $L$ is \emph{cofinal} if every member of
$L$ lies below some member of $C$, and \emph{coinitial} if every
member of $L$ lies above some member of $C$. The least sizes of such
subsets are the cofinality and coinitiality of $L$, respectively.
We call an increasing family $(f_\alpha)$ and a decreasing family
$(h_\beta)$ with $f_\alpha\le h_\beta$ a \emph{gap} when there is
no \emph{interpolant} $d$ satisfying $f_\alpha\le d\le h_\beta$
for every $\alpha,\beta$.

The next lemma converts the failure of Boolean completeness into a
gap in a maximal chain, whose two sides will provide the families
needed for the topological realization.

\begin{lemma}[A gap in a maximal Boolean chain]
\label{tcp:lem:boolean-gap}
If a Boolean algebra~$\mathfrak B$ is not complete, there exist infinite regular cardinals $\rho,\sigma$ and families
$(f_\alpha)_{\alpha<\rho}$, $(h_\beta)_{\beta<\sigma}$ in $\mathfrak B$ such that
\[
 f_0=0,\quad h_0=1,\quad
 f_\alpha\le f_{\alpha'}\le h_{\beta'}\le h_\beta
 \quad(\alpha\le\alpha',\ \beta\le\beta'),
\]
and no $d\in\mathfrak B$ satisfies
$f_\alpha\le d\le h_\beta$ for all $\alpha,\beta$.
\end{lemma}

\begin{proof}
We include a proof of the maximal-chain gap fact; see also~\cite[p.~344]{Monk}. Choose a family $(d_\xi)_{\xi<\eta}$ without a supremum, of least possible cardinality. Since finite joins exist, $\eta$ is infinite. For every $\alpha<\eta$, the partial join $c_\alpha=\bigvee_{\xi<\alpha}d_\xi$ exists by minimality. The chain $(c_\alpha)$ has no supremum, since its upper bounds are exactly those of $(d_\xi)$.

Extend it to a maximal chain~$M$ in~$\mathfrak B$, and put
\[
 Y=\{m\in M:m\le c_\alpha\text{ for some }\alpha<\eta\}, \qquad Z=M\setminus Y.
\]
Then $Y<Z$, $0\in Y$, and $1\in Z$. A greatest member~$y$ of~$Y$ would be an upper bound of the $c_\alpha$ lying below one of them, and hence would be a greatest member of that chain, a contradiction.

Suppose that~$Z$ has a least member~$z$, and let $b\in\mathfrak B$ be any upper bound of the $c_\alpha$. Every member of~$Y$ lies below $z\wedge b$, and every member of~$Z$ lies above it. Thus $z\wedge b$ is comparable with every member of~$M$, so maximality
puts it in~$M$. It cannot lie in~$Y$, because it is an upper bound of all the $c_\alpha$. Hence it lies in~$Z$, giving
$z\le z\wedge b\le b$. This would make~$z$ the supremum of the $c_\alpha$, another contradiction.

Let $\rho$ be the least cardinality of a cofinal subset of $Y$. It is infinite because $Y$ has no greatest member. Every subset of
$Y$ of cardinality less than $\rho$ has a strict upper bound in $Y$. Starting from a cofinal subset of size $\rho$, transfinite recursion therefore gives an increasing cofinal sequence of length $\rho$, with first member $0$. If $\rho$ were singular, restricting that sequence to a cofinal subset of the ordinal $\rho$ of cardinality $\operatorname{cf}(\rho)<\rho$ would contradict minimality. Thus $\rho$ is regular. The order-dual argument gives an infinite regular coinitiality $\sigma$ for $Z$, and a decreasing coinitial sequence starting with $1$. Denote the resulting sequences by $(f_\alpha)_{\alpha<\rho}$ and $(h_\beta)_{\beta<\sigma}$. Their cross inequalities follow from $Y<Z$. An interpolant~$d$ would lie above~$Y$ and below~$Z$, and would therefore belong to~$M$ by maximality. Membership in~$Y$ would make it the greatest member of~$Y$; membership in~$Z$ would make it the least member of~$Z$. Both alternatives are impossible.
\end{proof}

In the zero-dimensional case, a Boolean gap yields two increasing
families of clopens whose unions are disjoint but have intersecting
closures, since a clopen separator would provide an interpolant.

\begin{proposition}[The zero-dimensional case]
\label{tcp:prop:zero-dimensional}
Every compact Hausdorff zero-dimensional space~$K$ which is not extremally disconnected admits an inseparable pair indexed by
infinite regular cardinals. If~$K$ is also an $F$-space, at least one of the two cardinals is uncountable.
\end{proposition}

\begin{proof}
The Boolean algebra $\mathfrak B=\operatorname{Clop}(K)$ is not complete. Indeed, if it were complete and $O\subset K$ were open,
let $H$ be the supremum in~$\mathfrak B$ of all clopens contained in~$O$. These clopens cover~$O$, so $\overline O\subseteq H$.
If $t\in H\setminus\overline O$, zero-dimensionality gives a nonempty clopen $C\subseteq H\setminus\overline O$. Then
$H\setminus C$ is still an upper bound for all clopens contained in~$O$, contradicting the definition of~$H$. Thus $H=\overline O$, which would make~$K$ extremally disconnected.

Apply Lemma~\ref{tcp:lem:boolean-gap}. Regard its members as clopens and write $F_\alpha=f_\alpha$ and $G_\beta=K\setminus h_\beta$. These families are increasing and cross-disjoint. Suppose the closures of their unions were disjoint. For two disjoint closed subsets of a compact zero-dimensional space there is a clopen separator: cover the first set by clopen neighborhoods avoiding the second and take a finite subcover. Such a separator~$C$ would satisfy $F_\alpha\subseteq C\subseteq h_\beta$ for all~$\alpha,\beta$, contrary to the lemma. Therefore \eqref{tcp:eq:terminal-intersection} holds, and we may take $u_\alpha=\mathbf1_{F_\alpha}$, $v_\beta=\mathbf1_{G_\beta}$.

If $\rho=\sigma=\omega$ and~$K$ is an $F$-space, the two unions would be disjoint countable unions of cozero sets. Their closures would be disjoint by Lemma~\ref{tcp:lem:countable-cozero}, a contradiction.
\end{proof}

\subsection{Transfinite sign extension in the remaining case}

When clopen sets do not separate points, we instead extend prescribed
positive and negative values until the closures of the corresponding
regions meet, with the $F$-space property forcing this first
obstruction to have uncountable cofinality.

\begin{proposition}[The case of a non-zero-dimensional $F$-space]
\label{tcp:prop:non-zero-dimensional}
Every compact Hausdorff $F$-space~$K$ which is not zero-dimensional
admits an inseparable pair with $\rho=\sigma$ an uncountable regular
cardinal.
\end{proposition}

\begin{proof}
There exist distinct points $p,q\in K$ that cannot be separated by a clopen. Otherwise, given $p\in O$ with~$O$ open, for every
$q\in K\setminus O$ one could choose a clopen containing~$p$ and excluding~$q$. Compactness of $K\setminus O$ would give a finite
intersection of these clopens containing~$p$ and contained in~$O$, which would prove zero-dimensionality.

Set $h_0=0$ and use Urysohn's lemma to choose $h_1\in C(K,[-1,1])$ with $h_1(p)=1$ and $h_1(q)=-1$.
Fix a well-order of $C(K,[-1,1])$ and always choose the first admissible extension in that order. This specifies the following transfinite recursion. Suppose $(h_\alpha)_{\alpha<\gamma}$ has been defined, where $\gamma\ge2$,
and put
\[
 U_\gamma=\bigcup_{\alpha<\gamma}\{h_\alpha>0\},\qquad
 V_\gamma=\bigcup_{\alpha<\gamma}\{h_\alpha<0\}.
\]
If $\overline{U_\gamma}\cap\overline{V_\gamma}\ne\varnothing$, stop. Otherwise choose, by Urysohn's lemma,
$h_\gamma\in C(K,[-1,1])$ such that
\begin{equation}
 h_\gamma=1\text{ on }\overline{U_\gamma},\qquad
 h_\gamma=-1\text{ on }\overline{V_\gamma}.
 \label{tcp:eq:sign-extension}
\end{equation}
Inductively, the positive parts~$h_\alpha^+$ and negative parts~$h_\alpha^-:=(-h_\alpha)^+$ are increasing. At a point where
an earlier function is positive, every later function is~$1$; at a point where an earlier function is negative, every later
function is~$-1$. In particular, $U_\gamma,V_\gamma$ are disjoint at every stage, and $h_\alpha(p)=1$, $h_\alpha(q)=-1$ for $\alpha\ge1$.

For each $\alpha\ge1$, some $t\in K$ satisfies $0<|h_\alpha(t)|<1$. Otherwise the image of~$h_\alpha$ would be
contained in $\{-1,0,1\}$, and $\{h_\alpha=1\}$ would be a clopen separating~$p$ and~$q$. Equation~\eqref{tcp:eq:sign-extension} makes every later function take the value $\operatorname{sgn}(h_\alpha(t))$ at that point, so all the constructed functions with positive indices are distinct. They are also different from~$h_0$. The construction therefore cannot continue through all ordinals below $|C(K)|^+$: otherwise it would give an injection of that successor cardinal into $C(K)$. There is consequently a first stopping stage $\delta<|C(K)|^+$.

The stage~$\delta$ cannot be a successor, say $\delta=\alpha+1$: by nestedness, $U_\delta=\{h_\alpha>0\}$ and $V_\delta=\{h_\alpha<0\}$, which are disjoint cozero sets and hence have disjoint closures. Thus~$\delta$ is a limit. If $\operatorname{cf}(\delta)=\omega$, choose a countable cofinal sequence $(\alpha_n)$ in~$\delta$. Then $U_\delta=\bigcup_n\{h_{\alpha_n}>0\}$ and $V_\delta=\bigcup_n\{h_{\alpha_n}<0\}$ would again be disjoint cozero sets by Lemma~\ref{tcp:lem:countable-cozero}. This contradicts the stopping condition. Hence $\rho=\operatorname{cf}(\delta)>\omega$.

Every $\alpha<\delta$ satisfies $\alpha+1<\delta$. For $\alpha\ge1$, equation~\eqref{tcp:eq:sign-extension} gives the following inclusions; for $\alpha=0$ their left-hand sides are empty:
\[
 \{h_\alpha>0\}\subseteq\{h_{\alpha+1}=1\},\qquad
 \{h_\alpha<0\}\subseteq\{h_{\alpha+1}=-1\}.
\]
Consequently
\[
 U_\delta=\bigcup_{\alpha<\delta}\{h_\alpha=1\},\qquad
 V_\delta=\bigcup_{\alpha<\delta}\{h_\alpha=-1\}.
\]
Choose an increasing cofinal map $\xi\mapsto\gamma_\xi$ from~$\rho$ into~$\delta$ with $\gamma_0=0$, and set
$u_\xi=h_{\gamma_\xi}^+$, $v_\xi=h_{\gamma_\xi}^-$. The monotonicity of the parts implies cross-orthogonality, and
cofinality preserves the two terminal unions of level sets. Their closures intersect at the stopping stage. We have therefore
obtained Definition~\ref{tcp:def:inseparable-pair}, with $\sigma=\rho$, which is regular as a cofinality.
\end{proof}

\section{A single ordinal propagator and the BFPP characterization}
\label{sec:conclusion}
\label{tcp:subsec:main-theorem}

\subsection{Encoding the two channels}

The product notation used in the realization theorem is convenient for the estimates. We now encode both channels in a single ordinal and check the precise envelope rules.

\begin{proposition}[Ordinal encoding]
\label{tcp:prop:encoding}
Under the hypotheses of Theorem~\ref{tcp:thm:two-profile-realization}, put $\kappa=\max\{\rho,\sigma\}$. There are a nonempty closed bounded convex set $\mathcal D\subset\ell_\infty(\kappa)$, nonexpansive maps
\[
 \mathcal A:B_{C(K)}\to\mathcal D,
 \qquad \mathcal J:\mathcal D\to B_{C(K)},
\]
and a TCP $\mathscr P:\mathcal D\to\mathcal D$ such that $\operatorname{Fix}(\mathcal A\mathcal J\mathscr P)=\varnothing$.
\end{proposition}

\begin{proof}
For $a\in\mathcal M_\rho$, define its extension to~$\kappa$ by
\[
 \widetilde a_\alpha=
 \begin{cases}
  a_\alpha,&\alpha<\rho,\\
  \ell_\rho(a),&\rho\le\alpha<\kappa,
 \end{cases}
\]
and extend $b\in\mathcal M_\sigma$ in the same way. Extension is isometric on each profile space: its additional
coordinate differences are bounded by the original sup distance, while the original coordinates remain present. It is also affine
on increasing profiles because their tail functionals are affine. We shall use ordinal multiplication on the left by~$2$.
Every ordinal has a unique expression $\lambda+n$, where $\lambda=0$ or $\lambda$ is a nonzero limit, and $n<\omega$.
At a nonzero limit, writing $\lambda=\omega\cdot\xi$ gives $2\cdot\lambda=(2\cdot\omega)\cdot\xi=\lambda$.
In particular $2\cdot\kappa=\kappa$, and the pairs $\{2\cdot\alpha,2\cdot\alpha+1\}$, $\alpha<\kappa$, partition
the ordinals below $\kappa$. Define an affine isometry
\[
 E:\mathcal D_0\longrightarrow\ell_\infty(\kappa),
\]
by putting
\begin{equation}
 (E(a,b))_{2\cdot\alpha}=\widetilde a_\alpha,\quad
 (E(a,b))_{2\cdot\alpha+1}=4+\widetilde b_\alpha.
 \label{tcp:eq:encoding}
\end{equation}
Set
\begin{equation}
 \mathcal D=E(\mathcal D_0),\qquad
 \mathcal A=E\mathcal A_0, 
 \label{tcp:eq:encoded-maps}
\end{equation}
and
\begin{equation}
 \mathcal J=\mathcal J_0E^{-1},\qquad
 \mathscr P=EP_0E^{-1}.\label{tcp:eq:encoded-synthesis-propagator}
\end{equation}
Isometry and affinity give convexity and boundedness of~$\mathcal D$. It is closed because $\mathcal D_0$ is complete and~$E$ is an isometry. All three maps are nonexpansive, $\mathscr P(\mathcal D)\subseteq\mathcal D$, and
\[
 \mathcal A\mathcal J\mathscr P=E(\mathcal A_0\mathcal J_0P_0)E^{-1}
\]
has no fixed point.

It remains to verify that~$\mathscr P$ has the required ordinal rules. The delay commutes with extension:
\begin{equation}
 Q_\kappa\widetilde a=\widetilde{Q_\rho a},\qquad
 Q_\kappa\widetilde b=\widetilde{Q_\sigma b}.
 \label{tcp:eq:extension-delay}
\end{equation}
For example, if $\rho<\kappa$, the coordinate at~$\rho$ on the left is the supremum of the original channel, because~$\rho$ is
a limit ordinal. Every later coordinate has the same value, and $Q_\rho$ preserves that supremum. This proves the first identity;
the second is identical. For $x\in\mathcal D$, the initial values are $(\mathscr Px)_0=0$ and $(\mathscr Px)_1=4$. If $\alpha=\beta+1$, then
\begin{equation}
\begin{split}
 &(\mathscr Px)_{2\cdot\alpha}=x_{2\cdot\beta},\\[1.5mm]
 &(\mathscr Px)_{2\cdot\alpha+1}=x_{2\cdot\beta+1}.
\end{split}
 \label{tcp:eq:encoded-successors}
\end{equation}
For every nonzero limit~$\lambda<\kappa$, one has
$2\cdot\lambda=\lambda$, and
\begin{equation}
 \begin{aligned}
 &(\mathscr Px)_\lambda =\underline{\mathcal E}_\lambda(x|_{[0,\lambda)}),\\[1.5mm]
 &(\mathscr Px)_{\lambda+1}=\overline{\mathcal E}_\lambda(x|_{[0,\lambda)}).
 \end{aligned}
 \label{tcp:eq:encoded-limits}
\end{equation}
Indeed, before~$\lambda$ the two coordinate classes are cofinal. The first channel is increasing and takes values in~$[0,2]$;
the second is increasing and takes values in~$[4,6]$. The separation between these intervals allows the lower envelope to recover the first tail and the upper envelope to recover the second. It follows that
\[
 \liminf_{\xi\uparrow\lambda}x_\xi=\sup_{\alpha<\lambda}\widetilde a_\alpha,
\]
and
\[
 \limsup_{\xi\uparrow\lambda}x_\xi=4+\sup_{\alpha<\lambda}\widetilde b_\alpha.
\]
These are exactly the two limit values of the delayed channels in~\eqref{tcp:eq:extension-delay}. The rule at~$\lambda+1$ is a
nonexpansive successor rule depending only on the preceding input prefix. Thus~$\mathscr P$ satisfies Definition~\ref{tcp:def:propagator}.
\end{proof}

\subsection{The realization theorem and its consequences}

\begin{theorem}[Transfinite realization on a non-extremally disconnected
compact $F$-space]
\label{tcp:thm:main}
Let $K$ be a compact Hausdorff $F$-space which is not extremally disconnected. There exist an uncountable regular cardinal~$\kappa$, regarded as its initial ordinal, a nonempty closed bounded convex set $\mathcal D\subset\ell_\infty(\kappa)$, and nonexpansive maps
\[
 \mathcal A:B_{C(K)}\longrightarrow\mathcal D,
 \qquad
 \mathcal J:\mathcal D\longrightarrow B_{C(K)},
 \qquad
 \mathscr P:\mathcal D\longrightarrow\mathcal D
\]
such that~$\mathscr P$ is a TCP and
\[
 \operatorname{Fix}(\mathcal A\mathcal J\mathscr P)=\varnothing.
\]
In particular,
\[
 T=\mathcal J\mathscr P\mathcal A:
 B_{C(K)}\longrightarrow B_{C(K)}
\]
is nonexpansive and has no fixed point.
\end{theorem}

\begin{proof}
If~$K$ is zero-dimensional, use Proposition~\ref{tcp:prop:zero-dimensional}; otherwise use Proposition~\ref{tcp:prop:non-zero-dimensional}. In both cases there is an inseparable pair indexed by infinite regular cardinals $\rho,\sigma$, at least one of which is uncountable. Their maximum~$\kappa$ is therefore uncountable and regular. Apply Theorem~\ref{tcp:thm:two-profile-realization} and Proposition~\ref{tcp:prop:encoding}, followed by Proposition~\ref{tcp:prop:transfer}.
\end{proof}

\begin{corollary}
\label{tcp:cor:zero-dimensional}
If $K$ is compact Hausdorff, zero-dimensional, and not extremally disconnected, then $B_{C(K)}$ admits a fixed-point-free nonexpansive self-map, without an additional $F$-space assumption.
\end{corollary}

\begin{proof}
Apply Proposition~\ref{tcp:prop:zero-dimensional} and Theorem~\ref{tcp:thm:two-profile-realization}.
\end{proof}

\begin{corollary}[The BFPP characterization]
\label{tcp:cor:characterization}
For a compact Hausdorff space~$K$, the real Banach space~$C(K)$ has the BFPP if and only if~$K$ is extremally disconnected.
\end{corollary}

\begin{proof}
The implication from extremal disconnectedness, and the fact that the BFPP for~$C(K)$ forces~$K$ to be an $F$-space, are \eqref{eq:known-implications}, from~\cite[Corollary~2.4 and Theorem~3.2]{AJLCS}. If~$C(K)$ has the BFPP but~$K$ is not extremally disconnected,
Theorem~\ref{tcp:thm:main} applies to that compact $F$-space and gives a contradiction.
\end{proof}

\begin{remark}[Scalar fields]
The fixed-point-free construction in Theorem~\ref{tcp:thm:main} also applies to the complex unit ball: compose the real analysis
map with $f\mapsto\operatorname{Re}f$ and retain the real-valued synthesis. These maps are nonexpansive, and any fixed point would
be real-valued, contrary to the theorem.
\end{remark}


\section*{Data Availability}
Data sharing is not applicable to this article as no datasets were generated or analyzed during the current study.

\section*{Declaration of Competing Interest}
The author declares that he has no known competing financial interests or personal relationships that could have appeared to influence the work reported in this paper.

\section*{Declaration of generative AI and AI-assisted technologies in the manuscript preparation process}

During the preparation of this manuscript, OpenAI's Codex was used as a tool for mathematical exploration, drafting, language refinement, bibliographic checking, and the critical examination of arguments. The main use for this LLM was to assist in the special arguments used in some proofs. The author assumes full responsibility for verifying the statements and references and for the final content and mathematical accuracy of the manuscript.

\paragraph{Lean formalization}
A companion Lean 4 formalization of the mathematical results of this paper, using mathlib as its only library dependency,
is available at
\begin{center}
\url{https://github.com/Cleon-SB/ball-fixed-point-property-ck-spaces}
\end{center}
The repository contains the source code, build instructions, pinned dependency versions, and a correspondence between the
mathematical statements and their formal counterparts.

\end{document}